\documentclass[graybox]{svmult}

\usepackage{mathptmx}       
\usepackage{helvet}         
\usepackage{courier}        

\usepackage{makeidx}         
\usepackage{graphicx}        
\usepackage{multicol}        
\usepackage[bottom]{footmisc}

\makeindex             

\usepackage[utf8]{inputenc}
\usepackage[T1]{fontenc}
\usepackage{amsmath}
\usepackage{amssymb}
\usepackage{enumitem}
\usepackage{mathtools}
\usepackage[english]{babel}
\usepackage{euscript}
\usepackage{xcolor}
\usepackage{url}
\usepackage{xr} 

\newcommand{\BP}{\ensuremath{\EuScript P}}
\renewcommand{\leq}{\ensuremath{\leqslant}}
\renewcommand{\geq}{\ensuremath{\geqslant}}

\newcommand{\Frac}[2]{\displaystyle{\frac{#1}{#2}}} 
\newcommand{\scal}[2]{{\langle{{#1}\mid{#2}}\rangle}}

\newcommand{\menge}[2]{\big\{{#1}~\big |~{#2}\big\}} 
 
\newcommand{\KKK}{\ensuremath{\boldsymbol{\mathcal K}}}
\newcommand{\HHH}{\ensuremath{\boldsymbol{\mathcal H}}}
\newcommand{\GGG}{\ensuremath{{\boldsymbol{\mathcal G}}}}
\newcommand{\HH}{\ensuremath{{\mathcal H}}}
\newcommand{\GG}{\ensuremath{{\mathcal G}}}

\newcommand{\Sum}{\ensuremath{\displaystyle\sum}}

\newcommand{\emp}{\ensuremath{{\varnothing}}}
\newcommand{\Id}{\ensuremath{\operatorname{Id}}}

\newcommand{\RP}{\ensuremath{\left[0,+\infty\right[}}
\newcommand{\BL}{\ensuremath{\EuScript B}\,}
\newcommand{\SL}{\ensuremath{\EuScript S}\,}
\newcommand{\RPP}{\ensuremath{\left]0,+\infty\right[}}

\newcommand{\NN}{\ensuremath{\mathbb N}}
\newcommand{\weakly}{\ensuremath{\:\rightharpoonup\:}}
\newcommand{\exi}{\ensuremath{\exists\,}}
\newcommand{\ran}{\ensuremath{\text{\rm ran}\,}}
\newcommand{\zer}{\ensuremath{\text{\rm zer}\,}}
\newcommand{\pinf}{\ensuremath{{+\infty}}}

\newcommand{\dom}{\ensuremath{\text{\rm dom}\ }}
\newcommand{\Fix}{\ensuremath{\text{\rm Fix}\,}}

\newcommand{\gra}{\ensuremath{\text{\rm gra}\,}}
\newcommand{\inte}{\ensuremath{\text{\rm int }}}
\newcommand{\infconv}{\ensuremath{\mbox{\small$\,\square\,$}}}
\newcommand{\zeroun}{\ensuremath{\left]0,1\right[}}   
\newcommand{\rzeroun}{\ensuremath{\left]0,1\right]}}

\renewcommand\theenumii{(\alph{enumii})}

\begin{document}

\title{
Variable metric algorithms driven by averaged operators
}
\titlerunning{Variable metric algorithms driven by averaged operators.}
\author{Lilian E. Glaudin}
\institute{Sorbonne Université, Laboratoire Jacques-Louis Lions, 4 place
Jussieu, 75005 Paris, France	
\email{glaudin@ljll.math.upmc.fr}
}
%
%
\maketitle

\abstract{
The convergence of a new general variable metric algorithm based on
compositions of averaged operators is established. Applications to 
monotone operator splitting are presented.
}
\title{\sffamily\huge
}



\section{Introduction}
\label{sec:intro}

Iterations of averaged nonexpansive operators provide a synthetic
framework for the analysis of many algorithms in nonlinear
analysis, e.g., \cite{Bail78,Livre1,Cegi12,Comb04,Comb15}.
We establish the convergence of a new general variable metric 
algorithm based on compositions of averaged operators. These
results are applied to the analysis of the convergence of a new
forward-backward algorithm for solving the inclusion
\begin{equation}
0\in Ax+Bx,
\end{equation}
where $A$ and $B$ are maximally monotone operators 
on a real Hilbert space.
The theory of monotone operators is used in many applied mathematical
fields, including optimization \cite{Comb18_2}, 
partial differential equations and evolution inclusions
\cite{Brez73,Peyp10,Show97}, signal processing \cite{Comb11,Comb05},
and statistics and machine learning \cite{Comb18,Duch09,Jena11}.
In recent years, variants of the forward-backward algorithm with
variable metric have been proposed in
\cite{Comb13,Comb14,Salz17,Vu13}, as well as variants involving
overrelaxations~\cite{Comb15}. The goal of the present paper is to
unify these two approaches in the general context of iterations of
compositions of averaged operators. In turn, this provides new
methods to solve the problems studied in 
\cite{Alot14,Livre1,Bric11,Chamb11,Comb04,Comb12}.

The paper is organized as follows: Section~\ref{sec:not} presents
the background and notation. We establish the proof of the
convergence of the general algorithm in Section~\ref{sec:mainconv}.
Special cases are provided in Section~\ref{sec:app}. Finally, by
recasting these results in certain product spaces, 
we present and solve a general monotone inclusion in
Section~\ref{sec:composite}.

\section{Notation and background}
\label{sec:not}

Throughout this paper, $\HH$, $\GG$, and $(\GG_i)_{1\leq i\leq m}$ are real 
Hilbert spaces. We use $\scal{\cdot}{\cdot}$ to denote
the scalar product of a Hilbert space and $\|\cdot\|$
for the associated norm.
Weak and strong convergence are respectively denoted by $\weakly$ and
$\to$.
We denote by $\BL(\HH,\GG)$ the space of bounded linear operators 
from $\HH$ to $\GG$, and set $\BL(\HH)=\BL(\HH,\HH)$ and
$\SL(\HH)=\menge{L\in\BL(\HH)}{L=L^*}$, where $L^*$ denotes the
adjoint of $L$, and $\Id$ denotes the identity operator. 
The Loewner partial ordering on $\SL(\HH)$ is defined by
\begin{equation}
(\forall U\in\SL(\HH))(\forall V\in\SL(\HH))\quad
U\succcurlyeq V\quad\Leftrightarrow\quad(\forall x\in\HH)\quad
\scal{Ux}{x}\geq\scal{Vx}{x}.
\end{equation}
Let $\alpha\in\RPP$. We set
\begin{equation}
\BP_{\alpha}(\HH)=\menge{U\in\SL(\HH)}{U\succcurlyeq\alpha\Id},
\end{equation}
and we denote by $\sqrt{U}$ the square root of 
$U\in\BP_{\alpha}(\HH)$. Moreover, for every 
$U\in\BP_\alpha(\HH)$, we define a scalar product and 
a norm by
\begin{equation}
(\forall x\in\HH)(\forall y\in\HH)\quad\scal{x}{y}_U=\scal{Ux}{y}
\quad\text{and}\quad\|x\|_U=\sqrt{\scal{Ux}{x}},
\end{equation}
and we denote this Hilbert space by $(\HH,U)$.
Let $A\colon\HH\to 2^\HH$ be a set-valued operator.
We denote by 
$\dom A=\menge{x\in\HH}{Ax\neq\emp}$ the domain of $A$,
by $\gra A=\menge{(x,u)\in\HH\times\HH}{u\in Ax}$ the graph of $A$,
by $\ran A=\menge{u\in\HH}{(\exi x\in\HH)\;u\in Ax}$
the range of $A$, by 
$\zer A=\menge{x\in\HH}{0\in Ax}$ the set of zeros of $A$, 
and by $A^{-1}$ the inverse of $A$ which is the operator with 
graph $\menge{(u,x)\in\HH\times\HH}{u\in Ax}$. The resolvent of $A$ is
$J_A=(\Id+A)^{-1}$. Moreover, $A$ is monotone if
\begin{equation}
(\forall (x,u)\in\gra A)(\forall (y,v)\in\gra A)\quad
\scal{x-y}{u-v}\geq 0,
\end{equation}
and maximally monotone if there exists no monotone operator 
$B\colon\HH\to 2^{\HH}$ such that $\gra A\subset\gra B\neq\gra A$.
The parallel sum of $A\colon\HH\to 2^{\HH}$ and $B\colon\HH\to 2^{\HH}$ is 
\begin{equation}
A\infconv B=(A^{-1}+ B^{-1})^{-1}.
\end{equation} 
An operator $B\colon\HH\to 2^{\HH}$ is cocoercive with constant
$\beta\in\RPP$ if 
\begin{equation}
\label{e:coco}
(\forall x\in\HH)(\forall y\in\HH)\quad
\scal{x-y}{Bx-By}\geq\beta\|Bx-By\|^2.
\end{equation}
Let $C$ be a nonempty subset of $\HH$. The interior of $C$ is $\inte C$.
Finally, the set of summable sequences in $\RP$ is denoted by $\ell_+^1(\NN)$.

\begin{definition}
	Let $\mu\in\RPP$, let $U\in\BP_\mu(\HH)$, 
	let $D$ be a nonempty subset of $\HH$, let $\alpha\in\rzeroun$, and
	let $T\colon \HH\to\HH$ be an operator. Then $T$ is an $\alpha$-averaged
 operator on $(\HH,U)$ if
\begin{equation}
(\forall x\in\HH)(\forall y\in\HH)\quad \|Tx-Ty\|_U^2\leq
\|x-y\|^2_U-\dfrac{1-\alpha}{\alpha}\|Tx-x\|^2_U.
\end{equation}
If $\alpha=1$, $T$ is nonexpansive on $(\HH,U)$.
\end{definition}

\begin{lemma}\emph{\cite[Proposition~4.46]{Livre1}}
\label{l:compo}
Let $m\geq 1$ be an integer. For every $i\in\{1,\ldots,m\}$, let $T_i\colon
\HH\to \HH$ be averaged. Then $T_1\cdots T_m$ is averaged.
\end{lemma}

\begin{lemma}\emph{\cite[Proposition~4.35]{Livre1}}
\label{l:nonex}
Let $\mu\in\RPP$, let $U\in\BP_{\mu}(\HH)$, let $\alpha\in\rzeroun$,
and let $T$ be an $\alpha$-averaged operator on $(\HH,U)$. 
Then the operator $R=(1-1/\alpha)\Id+(1/\alpha)T$ is 
nonexpansive on $(\HH,U)$.
\end{lemma}

\begin{lemma}\emph{\cite[Lemma~5.31]{Livre1}}
\label{l:12}
Let $(\alpha_n)_{n\in\NN}$ and $(\beta_n)_{n\in\NN}$ be sequences in $\RP$, 
let $(\eta_n)_{n\in\NN}$ and $(\varepsilon_n)_{n\in\NN}$ be sequences in 
$\in\ell^1_+(\NN)$ such that 
\begin{equation}
(\forall n\in\NN)\quad
\alpha_{n+1}\leq (1+\eta_n)\alpha_n-\beta_n+\varepsilon_n.
\end{equation}
Then $(\beta_n)_{n\in\NN}\in\ell^1_+(\NN)$.
\end{lemma}	

\begin{lemma}\emph{\cite[Proposition~4.1]{Comb13}}
\label{l:11} 
Let $\alpha\in\RPP$, let $(W_n)_{n\in\NN}$ be in 
$\BP_{\alpha}(\HH)$, let $C$ be a nonempty subset of $\HH$, 
and let $(x_n)_{n\in\NN}$ be a sequence in $\HH$ such that 
\begin{multline}
\label{e:vmqf}
\big(\exi(\eta_n)_{n\in\NN}\in\ell_+^1(\NN)\big)
\big(\forall z\in C\big)\big(\exi(\varepsilon_n)_{n\in\NN}\in
\ell_+^1(\NN)\big)(\forall n\in\NN)\\
\|x_{n+1}-z\|^2_{W_{n+1}}\leq(1+\eta_n)
\|x_n-z\|^2_{W_n}+\varepsilon_n.
\end{multline}
Then $(x_n)_{n\in\NN}$ is bounded and, for every $z\in C$,
$(\|x_n-z\|_{W_n})_{n\in\NN}$ converges.
\end{lemma}

\begin{proposition}\emph{\cite[Theorem~3.3]{Comb13}}
\label{p:vmwf} 
Let $\alpha\in\RPP$, and let $(W_n)_{n\in\NN}$ and $W$ be 
operators in $\BP_{\alpha}(\HH)$ such that $W_n\to W$ pointwise, 
as is the case when 
\begin{equation}
\sup_{n\in\NN}\|W_n\|<\pinf\quad\text{and}\quad
(\exi (\eta_n)_{n\in\NN}\in\ell_+^1(\NN))
(\forall n\in\NN)\quad (1+\eta_n)W_n\succcurlyeq W_{n+1}.
\end{equation}
Let $C$ be a nonempty subset of 
$\HH$, and let $(x_n)_{n\in\NN}$ be a sequence in $\HH$ such 
that \eqref{e:vmqf} is satisfied. Then 
$(x_n)_{n\in\NN}$ converges weakly to a point in $C$ 
if and only if every weak sequential cluster point of 
$(x_n)_{n\in\NN}$ is in $C$.
\end{proposition}

\begin{proposition}\emph{\cite[Proposition~3.6]{Comb14}}
\label{p:vmstrong2}
Let $\alpha\in\RPP$, let $(\nu_n)_{n\in\NN}\in\ell_+^1(\NN)$, and
let $(W_n)_{n\in\NN}$ be a sequence in $\BP_{\alpha}(\HH)$
such that $\sup_{n\in\NN}\|W_n\|<\pinf$ and
$(\forall n\in\NN)$ $(1+\nu_n)W_{n+1}\succcurlyeq W_n$.
Furthermore, let $C$ be a subset of $\HH$ such that 
$\inte C\neq\emp$ and let $(x_n)_{n\in\NN}$ be a sequence in 
$\HH$ such that 
\begin{multline}
\big(\exi(\varepsilon_n)_{n\in\NN}\in\ell_+^1(\NN)\big)
\big(\exi(\eta_n)_{n\in\NN}\in\ell_+^1(\NN)\big)
(\forall x\in \HH)(\forall n\in\NN)\\
\|x_{n+1}-x\|^2_{W_{n+1}}\leq
(1+\eta_n)\|x_n-x\|^2_{W_n}+\varepsilon_n.
\end{multline}
Then $(x_n)_{n\in\NN}$ converges strongly.
\end{proposition}

\begin{proposition}\emph{\cite[Proposition~3.4]{Comb13}}
\label{p:vmdc} 
Let $\alpha\in\RPP$, let $(W_n)_{n\in\NN}$ be a sequence in 
$\BP_{\alpha}(\HH)$ such that $\sup_{n\in\NN}\|W_n\|<\pinf$, let 
$C$ be a nonempty closed subset of $\HH$, and let 
$(x_n)_{n\in\NN}$ be a sequence in $\HH$ such that 
\begin{multline}
\big(\exi(\varepsilon_n)_{n\in\NN}\in\ell_+^1(\NN)\big)
\big(\exi(\eta_n)_{n\in\NN}\in\ell_+^1(\NN)\big)(\forall z\in C)
(\forall n\in\NN)\\
\|x_{n+1}-z\|^2_{W_{n+1}}\leq(1+\eta_n)\|x_n-z\|^2_{W_n}+\varepsilon_n.
\end{multline}
Then $(x_n)_{n\in\NN}$ converges strongly to a point in $C$ if 
and only if $\varliminf d_C(x_n)=0$.
\end{proposition}

\begin{lemma}\emph{\cite[Lemma~3.1]{Comb14}}
\label{l:vmmat}
Let $\alpha\in\RPP$, let $\mu\in\RPP$, and let $A$ and $B$ be 
operators in $\SL(\HH)$ such that 
$\mu\Id\succcurlyeq A\succcurlyeq B\succcurlyeq\alpha\Id$. Then
the following hold:
\begin{enumerate}
\item
\label{l:vmmat-i}
$\alpha^{-1}\Id\succcurlyeq B^{-1}\succcurlyeq A^{-1}\succcurlyeq
\mu^{-1}\Id$. 
\item
\label{l:vmmat-ii}
$(\forall x\in\HH)$ $\scal{A^{-1}x}{x}\geq\|A\|^{-1}\|x\|^2$.
\item
\label{l:vmmat-iii}
$\|A^{-1}\|\leq\alpha^{-1}$.
\end{enumerate}
\end{lemma}

\section{Main convergence result}
\label{sec:mainconv}

We present our main result.
\begin{theorem}
\label{t:1}
Let $\alpha\in\RPP$, let $(\eta_n)_{n\in\NN}\in\ell^1_+(\NN)$, 
and let $(U_n)_{n\in\NN}$ be a sequence in $\mathcal{P}_\alpha(\HH)$ 
such that 
\begin{equation}
\label{e:t1u_n}
\mu=\sup_{n\in\NN}\|U_n\|<\pinf\quad\text{and}\quad
(\forall n\in\NN)\quad
(1+\eta_n)U_{n+1}\succcurlyeq U_n.
\end{equation}
Let $\varepsilon\in\left]0,1\right[$, let $m\geq 1$ be an 
integer, and let $x_0\in\HH$. 
For every $i\in\{1,\ldots,m\}$ and every $n\in\NN$, let 
$\alpha_{i,n}\in\zeroun$,
let $T_{i,n}\colon\HH\to\HH$ be $\alpha_{i,n}$-averaged
on $(\HH,U_n^{-1})$, let $\phi_n$ an averageness constant
of $T_{1,n}\cdots T_{m,n}$, let 
$\lambda_n\in\left]0,\phi_n\right[$,
and let $e_{i,n}\in\HH$. Iterate
\begin{equation}
\begin{array}{l}
\text{for}\;n=0,1,\ldots\\
\left\lfloor
\begin{array}{l}
	y_n=T_{1,n}\Big(T_{2,n}
\big(\cdots T_{m-1,n}(T_{m,n}x_n+e_{m,n})
+e_{m-1,n}\cdots\big)+e_{2,n}\Big)
+e_{1,n}\\
x_{n+1}=x_n+\lambda_n(y_n-x_n).
\end{array}
\right.\\
\end{array}
\label{e:t1main}
\end{equation}
Suppose that 
\begin{equation}
\label{h:t11}
S=\bigcap_{n\in\NN}\Fix(T_{1,n}\cdots T_{m,n})\neq\emp
\end{equation}
and
\begin{equation}
(\forall i\in\{1,\ldots,m\})\quad
\sum_{n\in\NN}\lambda_n\|e_{i,n}\|_{U_n^{-1}}<\pinf,
\end{equation}
and define
\begin{equation}
(\forall i\in\{1,\ldots,m\})(\forall n\in\NN)\quad T_{i+,n}=
\begin{cases}
T_{i+1,n}\cdots T_{m,n},&\text{if}\;\;i\neq m;\\
\Id,&\text{if}\;\;i=m.
\end{cases}
\end{equation}
Then the following hold:
\begin{enumerate}
\item
\label{t:1ii}
$\sum_{n\in\NN}\lambda_n(1/\phi_n
-\lambda_n)\|T_{1,n}\cdots T_{m,n}x_n-x_n\|^2_{U_n^{-1}}<\pinf$.
\item
\label{t:1iii}
Suppose that $(\forall n\in\NN)$ 
$\lambda_n\in\left]0,\varepsilon+(1-\varepsilon)/
\phi_n\right]$. Then
$(\forall x\in S)$
\begin{equation}
\underset{1\leq i\leq m}{\text{\rm max}}\Sum_{n\in\NN}
\Frac{\lambda_n(1-\alpha_{i,n})}{\alpha_{i,n}}
\left\|(\Id-T_{i,n})T_{i+,n}x_{n}-(\Id-T_{i,n})T_{i+,n}x
\right\|^2_{U_n^{-1}}<\pinf.
\end{equation}
\item
	\label{t:1iv}
$(x_n)_{n\in\NN}$ converges weakly to a point in $S$ if and only if
every weak sequential cluster point of $(x_n)_{n\in\NN}$ is in $S$.
In this case, the convergence is strong if $\inte S\neq\emp$.
\item
\label{t:1v}
$(x_n)_{n\in\NN}$ converges strongly to a point in $S$ if and only if 
$\varliminf d_S(x_n)=0$.

\end{enumerate}
\end{theorem}
\begin{proof}
Let $n\in\NN$ and let $x\in S$. Set 
\begin{equation}
\label{e:t1defTn}
T_n=T_{1,n}\cdots T_{m,n}
\end{equation}
and
\begin{equation}
\label{e:5thonf}
e_n=y_n-T_{n}x_n.
\end{equation}
Using the nonexpansiveness  on $(\HH,U_n^{-1})$ of the operators 
$(T_{i,n})_{1\leq i\leq m}$, we first derive from
\eqref{e:5thonf} that
\begin{align}
\|e_n\|_{U_n^{-1}}
\leq\sum_{i=1}^m\|e_{i,n}\|_{U_n^{-1}}.
\end{align}
Let us rewrite \eqref{e:t1main} as
\begin{equation}
\label{e:zn}
x_{n+1}=x_n+\lambda_n\big(T_nx_n+e_n-x_n\big),
\end{equation}
and set 
\begin{equation}
R_n=(1-1/\phi_n)\Id+(1/\phi_n)T_n
\quad\text{and}\quad\mu_n=\phi_n\lambda_n. 
\end{equation}
Then $\Fix R_n=\Fix T_n$ and, by Lemmas~\ref{l:compo} and~\ref{l:nonex}, 
$R_n$ is nonexpansive on $(\HH,U_n^{-1})$. Furthermore, \eqref{e:zn} can be 
written as
\begin{equation}
\label{e:t1_5}
x_{n+1}=x_n+\mu_n\big(R_nx_n-x_n\big)
+\lambda_ne_n,\quad\text{where}\quad\mu_n\in\zeroun. 
\end{equation}
Now set $z_n=x_n+\mu_n(R_nx_n-x_n)$. 
Since $x\in\Fix R_n$, we derive from \cite[Corollary~2.14]{Livre1} that
\begin{align}
	\label{e:t1_6}
\|z_n-x\|_{U_n^{-1}}^2
&=(1-\mu_n)\|x_n-x\|^2_{U_n^{-1}}+\mu_n\|R_nx_n-R_nx\|^2_{U_n^{-1}}\nonumber\\
&\qquad\qquad-\mu_n(1-\mu_n)\|R_nx_n-x_n\|^2_{U_n^{-1}}\\
\label{e:t1_6bis}
&\leq\|x_n-x\|_{U_n^{-1}}^2-\lambda_n(1/\phi_n-\lambda_n)\|T_nx_n-x_n\|^2_{U_n^{-1}}.
\end{align}
Hence, \eqref{e:t1_5}, \eqref{e:t1u_n}, and \eqref{e:t1_6bis} yield
\begin{align}
	\|x_{n+1}-x\|_{U^{-1}_{n+1}}
&\leq\sqrt{1+\eta_n}\|z_n-x\|_{U_n^{-1}}+\lambda_n\sqrt{1+\eta_n}\|e_n\|_{U_n^{-1}}
\label{e:t1_7}\\
&\leq\sqrt{1+\eta_n}\|x_n-x\|_{U_n^{-1}}+\lambda_n\sqrt{1+\eta_n}\|e_n\|_{U_n^{-1}}
\end{align}
and, since $\sum_{k\in\NN}\lambda_k\|e_k\|_{U_k}
<\pinf$, it follows from 
Lemma~\ref{l:11} that 
\begin{equation}
\label{e:nu}
\nu=\sum_{k\in\NN}\lambda_k\|e_k\|_{U_k^{-1}}+
2 \underset{k\in\NN}{\rm\text{sup}\,}\|x_k-x\|_{U_k^{-1}}<\pinf.
\end{equation}
Moreover, using \eqref{e:t1_7} and
\eqref{e:t1_6bis}
we write
\begin{align}
\label{e:12prem}
(1+\eta_n)^{-1}\|x_{n+1}-x\|_{U^{-1}_{n+1}}^2
&\leq\|z_n-x\|^2_{U_n^{-1}}+(2\|z_n-x\|_{U_n^{-1}}+\lambda_n\|e_n\|_{U_n^{-1}})
\lambda_n\|e_n\|_{U_n^{-1}}\\
&\leq\|x_n-x\|^2_{U_n^{-1}}-\lambda_n(1/\phi_n-\lambda_n)
\|T_nx_n-x_n\|^2_{U_n^{-1}}\nonumber\\
&\qquad\qquad+\nu\lambda_n\|e_n\|_{U_n^{-1}}.
\label{e:12}
\end{align}

\ref{t:1ii}: This follows from 
\eqref{e:12}, \eqref{e:t1defTn}, \eqref{h:t11}, \eqref{e:nu}, and Lemma~\ref{l:12}.

\ref{t:1iii}:
We apply the definition of averageness of the operators
$(T_{i,n})_{1\leq i\leq m}$ to obtain
\begin{align}
\label{e:t1_9}
\|T_nx_n-x\|^2_{U_n^{-1}}
&=\left\|T_{1,n}\cdots T_{m,n}x_n-T_{1,n}\cdots T_{m,n}x\right\|_{U_n^{-1}}^2
\nonumber\\
&\leq\|x_n-x\|_{U_n^{-1}}^2\nonumber\\
&\quad-\sum_{i=1}^{m}\dfrac{1-\alpha_{i,n}}{\alpha_{i,n}}
\left\|(\Id-T_{i,n})T_{i+,n}x_n-(\Id-T_{i,n})T_{i+,n}x\right\|_{U_n^{-1}}^2.
\end{align}
Note also that 
\begin{eqnarray}
\lambda_n\leq\varepsilon+\dfrac{1-\varepsilon}{\phi_n}
&\Rightarrow&
\dfrac{1}{\varepsilon}\lambda_n\leq(\dfrac{1}{\varepsilon}-1)
\dfrac{1}{\phi_n}
\nonumber\\
&\Leftrightarrow&
\lambda_n-1\leq\bigg(\dfrac{1}{\varepsilon}-1\bigg)
\bigg(\dfrac{1}{\phi_n}-\lambda_n\bigg).
\end{eqnarray}
Thus \eqref{e:12prem}, the definition of $z_n$, and \cite[Corollary~2.14]{Livre1} yield
\begin{align}
(1+\eta_n)^{-1}\|x_{n+1}-x\|_{U_{n+1}^{-1}}^2
&\leq\|(1-\lambda_n)(x_n-x)+\lambda_n(T_nx_n-x)\|^2_{U_n^{-1}}
+\nu\lambda_n\|e_n\|_{U_n^{-1}}
\nonumber\\
&=(1-\lambda_n)\|x_n-x\|_{U_n^{-1}}^2+\lambda_n\|T_nx_n-x\|_{U_n^{-1}}^2
\nonumber\\
&\qquad+\lambda_n(\lambda_n-1)\|T_nx_n-x_n\|_{U_n^{-1}}^2
+\nu\lambda_n\|e_n\|_{U_n^{-1}}\nonumber
\\
\label{e:t1_8bis}
&\leq(1-\lambda_n)\|x_n-x\|_{U_n^{-1}}^2+\lambda_n\|T_nx_n-x\|_{U_n^{-1}}^2
+\varepsilon_n,
\end{align}
where
\begin{equation}
\varepsilon_n=\lambda_n\bigg(\dfrac{1}{\varepsilon}-1\bigg)
\bigg(\dfrac{1}{\alpha_n}-\lambda_n\bigg)
\|T_nx_n-x_n\|_{U_n^{-1}}^2+\nu\lambda_n\|e_n\|_{U_n^{-1}}.
\end{equation}
Now set
\begin{equation}
\beta_n=\lambda_n
\underset{1\leq i\leq m}{\text{\rm max}}
\bigg(\dfrac{1-\alpha_{i,n}}{\alpha_{i,n}}
\left\|(\Id-T_{i,n})T_{i+,n}x_n-(\Id-T_{i,n})T_{i+,n}x\right\|_{U_n^{-1}}^2
\bigg).
\end{equation}
On the one hand, it follows from \ref{t:1ii}, 
\eqref{e:nu}, and \eqref{h:t11} that
\begin{equation}
\label{e:t1_10}
\sum_{k\in\NN}\varepsilon_k<\pinf.
\end{equation}
On the other hand, combining \eqref{e:t1_9} and 
\eqref{e:t1_8bis}, we obtain
\begin{equation}
	\label{e:t1_11}
(1+\eta_n)^{-1}\|x_{n+1}-x\|_{U_{n+1}^{-1}}^2
\leq\|x_n-x\|_{U_n^{-1}}^2-\beta_n+\varepsilon_n.
\end{equation}
Consequently, Lemma~\ref{l:12} implies that
$\sum_{k\in\NN}\beta_k<\pinf$. 

\ref{t:1iv}--\ref{t:1v}:
The results follow from \eqref{e:t1_11},
\eqref{e:t1_10}, and Proposition~\ref{p:vmwf} for the weak convergence,
and Propositions~\ref{p:vmstrong2} and~\ref{p:vmdc} 
for the strong convergence.
\end{proof}

\begin{remark}
Suppose that 
$(\forall n\in\NN)$ $U_n=\Id$ and $\lambda_n\leq
(1-\varepsilon)(1/\phi_n+\varepsilon)$. Then Theorem~\ref{t:1} reduces
to \cite[Theorem~3.5]{Comb15} which itself extends \cite[Section~3]{Comb04}
in the case $(\forall n\in\NN)$ $\lambda_n\leq 1$.
As far as we know, it is the first inexact overrelaxed 
variable metric algorithm based on averaged operators.
\end{remark}

\section{Applications to the forward-backward algorithm}
\label{sec:app}

A special case of Theorem~\ref{t:1} of interest is the following.

\begin{corollary}
\label{c:1}
Let $\alpha\in\RPP$, let $(\eta_n)_{n\in\NN}\in\ell^1_+(\NN)$, 
and let $(U_n)_{n\in\NN}$ be a sequence in $\mathcal{P}_\alpha(\HH)$ 
such that
\begin{equation}
\label{e:c1u_n}
\mu=\sup_{n\in\NN}\|U_n\|<\pinf\quad\text{and}\quad
(\forall n\in\NN)\quad
(1+\eta_n)U_{n+1}\succcurlyeq U_n.
\end{equation}
Let $\varepsilon\in\left]0,1\right[$ and let 
$x_0\in\HH$. For every $n\in\NN$, let 
$\alpha_{1,n}\in\left]0,1/(1+\varepsilon)\right]$, let
$\alpha_{2,n}\in\left]0,1/(1+\varepsilon)\right]$, 
let $T_{1,n}\colon\HH\to\HH$ be $\alpha_{1,n}$-averaged on $(\HH,U_n^{-1})$, 
let $T_{2,n}\colon\HH\to\HH$ be $\alpha_{2,n}$-averaged on $(\HH,U_n^{-1})$, 
let $e_{1,n}\in\HH$, and let $e_{2,n}\in\HH$. In addition, for
every $n\in\NN$, let
\begin{equation}
\label{e:c1_1}
\lambda_n\in\left[\varepsilon,
\varepsilon+\dfrac{1-\varepsilon}{\phi_n}\right], 
\quad\text{where}\quad
\phi_n=\dfrac{\alpha_{1,n}+\alpha_{2,n}-2\alpha_{1,n}\alpha_{2,n}}
{1-\alpha_{1,n}\alpha_{2,n}}, 
\end{equation}
and iterate 
\begin{equation}
x_{n+1}=x_n+\lambda_n\Big(T_{1,n}\big(T_{2,n}x_n+
e_{2,n}\big)+e_{1,n}-x_n\Big).
\end{equation}
Suppose that 
\begin{equation}
S=\bigcap_{n\in\NN}\Fix(T_{1,n} T_{2,n})\neq\emp,\quad
\Sum_{n\in\NN}\lambda_n\|e_{1,n}\|<\pinf,
\quad\text{and}\quad
\Sum_{n\in\NN}\lambda_n\|e_{2,n}\|<\pinf.
\end{equation}
Then the following hold:
\begin{enumerate}
\item
\label{c:1i0}
$\sum_{n\in\NN}\|T_{1,n}T_{2,n}x_n-x_n\|^2<\pinf$.
\item
\label{c:1i}
$(\forall x\in S)$ $\sum_{n\in\NN}
\|T_{1,n}T_{2,n}x_n-T_{2,n}x_n+T_{2,n}x-x\|^2<\pinf$.
\item
	\label{c:1ii}
$(\forall x\in S)$
$\sum_{n\in\NN}\|T_{2,n}x_n-x_n-T_{2,n}x+x\|^2<\pinf$.
\item
\label{c:1iii}
Suppose that every weak sequential cluster point of 
$(x_n)_{n\in\NN}$ is in $S$. Then $(x_n)_{n\in\NN}$ converges 
weakly to a point in $S$, and the convergence is strong if 
$\inte S\neq\emp$. 
\item
\label{c:1iv}
$(x_n)_{n\in\NN}$ converges strongly to a point in $S$ if and only if 
$\varliminf d_S(x_n)=0$.
\end{enumerate}
\end{corollary}
\begin{proof}
For every $n\in\NN$,
\begin{equation}
\label{e:23d}
\dfrac{1}{\sqrt{\mu}}\|e_{1,n}\|\leq\|e_{1,n}\|_{U_n^{-1}}
\quad\text{and}\quad
\dfrac{1}{\sqrt{\mu}}\|e_{2,n}\|\leq\|e_{2,n}\|_{U_n^{-1}},
\end{equation}
and $T_{1,n}T_{2,n}$ is $\phi_n$-averaged by
\cite[Proposition~4.44]{Livre1}. Thus,
we apply Theorem~\ref{t:1} with $m=2$. 

\ref{c:1i0}--\ref{c:1ii}:
This follows from
Theorem~\ref{t:1}\ref{t:1ii} that
\begin{equation}
\label{e:c11}
(\forall x\in S)\quad
\begin{cases}
\Sum_{n\in\NN}
\Frac{\lambda_n(1-\alpha_{1,n})}{\alpha_{1,n}}
\left\|(\Id-T_{1,n})T_{2,n}x_n-(\Id-T_{1,n})T_{2,n}x
\right\|^2_{U_n^{-1}}<\pinf\\
\Sum_{n\in\NN}\Frac{\lambda_n(1-\alpha_{2,n})}{\alpha_{2,n}}
\left\|(\Id-T_{2,n})x_n-(\Id-T_{2,n})x
\right\|^2_{U_n^{-1}}<\pinf\\
\Sum_{n\in\NN}\lambda_n\Big(\dfrac{1}{\phi_n}-\lambda_n\Big)
\left\|T_{1,n}T_{2,n}x_n-x_n
\right\|^2_{U_n^{-1}}<\pinf.
\end{cases}
\end{equation}
However, we derive from the assumptions that 
\begin{equation}
\label{e:c12}
(\forall x\in S)(\forall n\in\NN)\quad
\begin{cases}
T_{1,n}T_{2,n}x=x\\
\Frac{\lambda_n(1-\alpha_{1,n})}{\alpha_{1,n}}\geq\varepsilon^2\\
\Frac{\lambda_n(1-\alpha_{2,n})}{\alpha_{2,n}}\geq\varepsilon^2\\
	\lambda_n\Big(\dfrac{1}{\phi_n}-\lambda_n\Big)\geq
	\varepsilon\dfrac{1-\phi_n}{\phi_n}\geq
	\dfrac{2\varepsilon^2}{2\varepsilon+1}.
\end{cases}
\end{equation}
Combining \eqref{e:c1u_n}, \eqref{e:c11} and \eqref{e:c12} completes the
proof.

\ref{c:1iii}--\ref{c:1iv}: It follows
from  Theorem~\ref{t:1}\ref{t:1iv}--\ref{t:1v}.
\end{proof}

\begin{remark}
This corollary is a variable metric version of
\cite[Corollary~4.1]{Comb15} where
$(\forall n\in\NN)$ $U_n=\Id$ and $\lambda_n\leq
(1-\varepsilon)(1/\phi_n+\varepsilon)$.
\end{remark}

We recall the definition of a demiregular operator. See \cite{Atto10} for
examples of demiregular operators.

\begin{definition}{\rm\cite[Definition~2.3]{Atto10}}
\label{d:demireg}
An operator $A\colon\HH\to 2^{\HH}$ is \emph{demiregular} at
$x\in\dom A$ if, for every sequence $((x_n,u_n))_{n\in\NN}$ in 
$\gra A$ and every $u\in Ax$ such that $x_n\weakly x$ and 
$u_n\to u$, we have $x_n\to x$.
\end{definition}

\begin{proposition}
\label{p:fbop}
Let $\alpha\in\RPP$, let $U\in\BP_\alpha(\HH)$, let $A\colon \HH\to 2^\HH$ 
be a maximally monotone operator, let $\beta\in\RPP$, 
let $\gamma\in\left]0,2\beta/\|U\|\right]$, and let $B$ a $\beta$-cocoercive operator.
Then the following hold:
\begin{enumerate}
	\item 
		\label{p:pjabi}	
		$J_{\gamma UA}$ is a $1/2$-averaged operator
		on $(\HH,U^{-1})$.
\item 
		\label{p:pjabii}	
	$\Id-\gamma UB$ is a $\gamma\|U\|/(2\beta)$-averaged 
	 operator on $(\HH,U^{-1})$.
\end{enumerate}
\end{proposition}
\begin{proof}
	\ref{p:pjabi}: \cite[Lemma~3.7]{Comb14}.

	\ref{p:pjabii}: We derive from 
	\eqref{e:coco} and Lemma~\ref{l:vmmat}\ref{l:vmmat-iii} that
for every $x\in\HH$ and for every $y\in\HH$	
\begin{align}
\scal{x-y}{UBx-UBy}_{U^{-1}}
&=\scal{x-y}{Bx-By}\nonumber\\
&\geq \beta\scal{Bx-By}{Bx-By\nonumber}\\
&=\beta\scal{U^{-1}(UBx-UBy)}{UBx-UBy}_{U^{-1}}\nonumber\\
&\geq \|U\|^{-1}\beta\|UBx-UBy\|^2_{U^{-1}}.
\end{align}
Thus, for every $x\in\HH$ and for every $y\in\HH$	
\begin{align}
\|(x-\gamma UBx)-(y-\gamma UBy)\|^2_{U^{-1}}
&=\|x-y\|_{U^{-1}}^2
+\|\gamma UBx-\gamma UBy\|_{U^{-1}}^2\nonumber\\
&\qquad\qquad-2\gamma\scal{x-y}{UBx-UBy}_{U^{-1}}\\
&\leq\|x-y\|_{U^{-1}}^2\nonumber\\
&\qquad-\gamma(2\beta/\|U\|-\gamma)\|UBx-UBy\|^2_{U^{-1}},
\end{align}
which concludes the proof.
\end{proof}

Next, we introduce a new variable metric forward-backward splitting 
algorithm.

\begin{proposition}
\label{p:fb}
Let $\beta\in\RPP$, let $\varepsilon\in\left]0,\min\{1/2,\beta\}\right[$, 
let $\alpha\in\RPP$, let $(\eta_n)_{n\in\NN}\in\ell^1_+(\NN)$, 
and let $(U_n)_{n\in\NN}$ be a sequence in $\mathcal{P}_\alpha(\HH)$ 
such that
\begin{equation}
	\mu=\sup_{n\in\NN}\|U_n\|<\pinf
	\quad\text{and}\quad
(\forall n\in\NN)\quad
(1+\eta_n)U_{n+1}\succcurlyeq U_n.
\end{equation}
Let $x_0\in\HH$, let $A\colon\HH\to 2^{\HH}$ be maximally monotone, and let 
$B\colon\HH\to\HH$ be $\beta$-cocoercive.
Furthermore, let $(a_n)_{n\in\NN}$ and $(b_n)_{n\in\NN}$ be sequences in $\HH$
such that $\sum_{n\in\NN}\|a_n\|<\pinf$ and 
$\sum_{n\in\NN}\|b_n\|<\pinf$. Suppose that $\zer(A+B)\neq\emp$
and, for every $n\in\NN$, let 
\begin{equation}
\label{e:gamlam}
\gamma_n\in\left[\varepsilon,\dfrac{2\beta}{(1+\varepsilon)\|U_n\|}\right]\quad\text{and}\quad
\lambda_n\in\left[\varepsilon,1+(1-\varepsilon)
\bigg(1-\dfrac{\gamma_n\|U_n\|}{2\beta}\bigg)\right],
\end{equation}
and iterate 
\begin{equation}
\label{e:fb}
x_{n+1}=x_n+\lambda_n\Big(J_{\gamma_nU_n A}
\big(x_n-\gamma_nU_n(Bx_n+b_n)\big)+a_n-x_n\Big).
\end{equation}
Then the following hold:
\begin{enumerate}
\item
\label{p:fbi}
$\sum_{n\in\NN}\|J_{\gamma_nU_n A}
(x_n-\gamma_nU_nBx_n)-x_n\|^2<\pinf$.
\item
\label{p:fbii}
Let $x\in\zer(A+B)$. Then $\sum_{n\in\NN}\|Bx_n-Bx\|^2<\pinf$.
\item
\label{p:fbiii}
$(x_n)_{n\in\NN}$ converges weakly to a point in $\zer(A+B)$.
\item 
	\label{p:fbiv}	
	Suppose that one of the following holds:
	\begin{enumerate}
		\item $A$ is demiregular at every point in $\zer(A+B)$.
		\item $B$ is demiregular at every point in $\zer(A+B)$.
		\item $\inte S\neq \emp$.
	\end{enumerate}
Then $(x_n)_{n\in\NN}$ converges strongly to a point in $\zer (A+B)$.
\end{enumerate}
\end{proposition}
\begin{proof}
We apply Corollary~\ref{c:1}. Set 
\begin{equation}
(\forall n\in\NN)\;\;
\label{e:pfb_1}
T_{1,n}=J_{\gamma_nU_n A},\quad T_{2,n}=\Id-\gamma_nU_n B,\quad 
e_{1,n}=a_n,\quad\text{and}\quad e_{2,n}=-\gamma_nU_nb_n.
\end{equation}
Then, for every $n\in\NN$, $T_{1,n}$ is $\alpha_{1,n}$-averaged
on $(\HH,U_n^{-1})$ with $\alpha_{1,n}=1/2$ 
and $T_{2,n}$ is $\alpha_{2,n}$-averaged on $(\HH,U_n^{-1})$ with 
$\alpha_{2,n}=\gamma_n\|U_n\|/(2\beta)$ by Proposition~\ref{p:fbop}.
Moreover, for every $n\in\NN$,  
\begin{equation}
\label{e:fb2}
\phi_n=\dfrac{\alpha_{1,n}+\alpha_{2,n}-2\alpha_{1,n}\alpha_{2,n}}
{1-\alpha_{1,n}\alpha_{2,n}}=\dfrac{2\beta}{4\beta-\gamma_n\|U_n\|} 
\end{equation}
and, therefore, \eqref{e:gamlam} yields
\begin{equation}
\label{e:fb3}
\lambda_n\in\left[\varepsilon,\varepsilon+\dfrac{1-\varepsilon}{\phi_n}\right].
\end{equation}
Hence, we derive from \eqref{e:fb2} and \eqref{e:fb3} that 
$(\forall n\in\NN)$ $\lambda_n\leq 2+\varepsilon$. Consequently,
\begin{equation}
\begin{cases}
\sum_{n\in\NN}\lambda_n\|e_{1,n}\|=(2+\varepsilon)
\sum_{n\in\NN}\|a_n\|<\pinf\\
\sum_{n\in\NN}\lambda_n\|e_{2,n}\|\leq
2\beta(2+\varepsilon)\mu\alpha^{-1}\sum_{n\in\NN}\|b_n\|<\pinf. 
\end{cases}
\end{equation}
Furthermore, it follows from \cite[Proposition~26.1(iv)]{Livre1} that 
\begin{equation}
(\forall n\in\NN)\quad S=\zer(A+B)=\Fix(T_{1,n} T_{2,n})\neq\emp.
\end{equation}
Hence, the assumptions of Corollary~\ref{c:1} are satisfied.

\ref{p:fbi}: This is a consequence of 
Corollary~\ref{c:1}\ref{c:1i0} and
\eqref{e:pfb_1}.

\ref{p:fbii}: 
Corollary~\ref{c:1}\ref{c:1i}, \eqref{e:pfb_1}, and
Lemma~\ref{l:vmmat}\ref{l:vmmat-iii} yield 
\begin{align}
\sum_{n\in\NN}\|Bx_n-Bx\|^2
&=\sum_{n\in\NN}\gamma_n^{-2}\|U_n^{-1}(T_{2,n}x_n-x_n-T_{2,n}x+x)\|^2
\nonumber\\
&\leq\dfrac{1}{\varepsilon^{2}\alpha^2}\sum_{n\in\NN}\|T_{2,n}x_n-x_n-T_{2,n}x+x\|^2
\nonumber\\
&<\pinf.
\end{align}

\ref{p:fbiii}: 
Let $(k_n)_{n\in\NN}$ be a strictly
increasing sequence in $\NN$ and let $y\in\HH$ be such that 
$x_{k_n}\weakly y$. 
In view of Corollary~\ref{c:1}\ref{c:1iii}, 
it remains to show that $y\in\zer(A+B)$. 
Set
\begin{equation}
(\forall n\in\NN)\quad
\begin{cases}
y_n=J_{\gamma_n U_nA}(x_n-\gamma_nU_nBx_n)\\
u_n=\gamma_n^{-1}U_n^{-1}(x_n-y_n)-Bx_n\\
v_n = Bx_n
\end{cases}
\end{equation}
and let $z\in\zer(A+B)$.  Hence, we derive from \ref{p:fbi} that $y_n-x_n\to 0$.
Then $y_{k_n}\weakly y$ and by \ref{p:fbii} $Bx_{n}\to Bz$. 
Altogether, $y_{k_n}\weakly y$, $v_{k_n}\weakly Bz$,
$y_{k_n}-x_{k_n}\to 0$,
$u_{k_n}+v_{k_n}\to 0$, and, for every $n\in\NN$,
$u_{k_n}\in Ay_{k_n}$ and
$v_{k_n}\in Bx_{k_n}$.
It therefore follows from \cite[Lemma~4.5(ii)]{Siopt17} 
that $y\in\zer(A+B)$.

\ref{p:fbiv}: The proof is the same that in
\cite[Proposition~4.4(iv)]{Comb15}.
\end{proof}

\begin{remark}
Suppose that $(\forall n\in\NN)$ $U_n=\Id$ and 
$\lambda_n\leq (1-\varepsilon)(1/\phi_n+\varepsilon)$. Then
Proposition~\ref{p:fb} captures \cite[Proposition~4.4]{Comb15}. 
Now suppose that $(\forall n\in\NN)$
$\lambda_n\leq 1$. Then Proposition~\ref{p:fb} captures 
\cite[Theorem~4.1]{Comb14}.
\end{remark}

Using the averaged operators framework allows us to obtain
an extended forward-backward splitting algorithm in Euclidean spaces. 

\begin{example}
	Let $\alpha\in\RPP$, let $(\eta_n)_{n\in\NN}\in\ell^1_+(\NN)$, 
and let $(U_n)_{n\in\NN}$ be a sequence in $\mathcal{P}_\alpha(\HH)$ 
such that
\begin{equation}
\mu=\sup_{n\in\NN}\|U_n\|<\pinf\quad\text{and}\quad
(\forall n\in\NN)\quad
(1+\eta_n)U_{n+1}\succcurlyeq U_n.
\end{equation}
Let $\varepsilon\in\left]0,1/2\right[$,
let $A\colon \HH\to 2^\HH$ be a maximally monotone operator, let
$\beta\in\RPP$, let $B$ a $\beta$-cocoercive operator,
and let
$(\gamma_n)_{n\in\NN}$ and $(\mu_n)_{n\in\NN}$ be sequences in
$[\varepsilon,\pinf[$ 
such that
\begin{equation}
	\label{h:e11}
\phi_n=\dfrac{2\mu_n\beta}{4\beta-\|U_n\|\gamma_n}\leq 1-\varepsilon.
\end{equation}
Let $x_0\in\HH$ and iterate
\begin{equation}
\label{}
(\forall n\in\NN)\quad x_{n+1}=x_n+\mu_n\Big(
J_{\gamma_n U_nA}(x_n-\gamma_nU_nB)-x_n\Big).
\end{equation}
Suppose that $\HH$ is finite-dimensional and that $\zer(A+B)\neq\emp$.
Then $(x_n)_{n\in\NN}$ converges to a point in $\zer(A+B)$.
\end{example}
\begin{proof}
Set $(\forall n\in\NN)$ 
$T_n=\Id+\mu_n(J_{\gamma_n U_nA}(\Id-\gamma_n U_n B)-\Id)$. Remark
that, for every $n\in\NN$, $T_n$ is $\phi_n$-averaged. 
Hence we apply Theorem~\ref{t:1} with $m=1$ and $\lambda\equiv 1$. 
It follows from Theorem~\ref{t:1}\ref{t:1ii} and \eqref{h:e11}
that $T_nx_n-x_n\to 0$. Since $\HH$ is finite-dimensional,
the claim follows from
Theorem~\ref{t:1}\ref{t:1iv}.
\end{proof}

\begin{remark}
An underrelaxation or an appropriate choice of the metric of 
the algorithm allows us to exceed 
the classical bound $2/\beta$ for $(\gamma_n)_{n\in\NN}$.
For instance, the parameters $\gamma_n\equiv 2.99/\beta$,
$\mu_n\equiv 1/2$, and $U_n\equiv\Id$ satisfy the
assumptions.
\end{remark}

\section{A composite monotone inclusion problem}
\label{sec:composite}

We study the composite monotone inclusion presented in \cite{Comb12}.

\begin{problem}
\label{prob:1}
Let $z\in\HH$, let $A\colon\HH\to 2^{\HH}$ be maximally monotone, 
let $\mu\in\RPP$, let $C\colon\HH\to\HH$ be $\mu$-cocoercive, 
and let $m$ be a strictly positive integer.
For every $i\in\{1,\ldots, m\}$, let $r_i\in\GG_i$, 
let $B_i\colon \GG_i\to2^{\GG_i}$ be maximally monotone, 
let $\nu_i\in\RPP$,
let $D_i\colon \GG_i \to 2^{\GG_i}$ be maximally monotone and
$\nu_i$-strongly monotone,
and suppose that $0\neq L_i\in \BL(\HH,\GG_i)$. 
The problem is to find $\overline{x}\in\HH$ such that
\begin{equation}
\label{e:prob11}
z\in A\overline{x}+\sum_{i=1}^mL^{*}_i
\big((B_i\infconv D_i)(L_i\overline{x}-r_i)\big)+C\overline{x},
\end{equation}
the dual problem of which is to find
$\overline{v}_1\in\GG_1,\ldots,
\overline{v}_m\in\GG_m$ such that
\begin{equation}
\label{e:prob12}
(\exi x\in\HH)\quad
\begin{cases}
z-\sum_{i=1}^m L_{i}^*\overline{v}_i \in Ax+Cx\\
(\forall i\in\{1,\ldots,m\})\;\overline{v}_i\in 
(B_i\infconv D_i)(L_ix-r_i).
\end{cases}
\end{equation}
\end{problem}

The following corollary is an overrelaxed version of
\cite[Corollary~6.2]{Comb14}.
\begin{corollary}
\label{c:fbpd}
In Problem~\ref{prob:1}, suppose that 
\begin{equation}
z\in\ran
\bigg(A+\sum_{i=1}^mL^{*}_i
\big((B_i\infconv D_i)(L_i\cdot-r_i)\big)+C\bigg),
\end{equation}
and set
\begin{equation}
\beta=\min\{\mu,\nu_1,\ldots,\nu_m\}.
\end{equation}
Let $\varepsilon\in\left]0,\min\{1,\beta\}\right[$,
let $\alpha\in\RPP$, let $(\lambda_n)_{n\in\NN}$ be a sequence in 
$\RPP$, let $x_0\in\HH$, let 
$(a_n)_{n\in\NN}$ and $(c_n)_{n\in\NN}$ be absolutely summable 
sequences in $\HH$, and let $(U_n)_{n\in\NN}$ be a sequence 
in $\BP_{\alpha}(\HH)$ such that 
$(\forall n\in\NN)\; U_{n+1}\succcurlyeq U_n$.
For every $i\in\{1,\ldots,m\}$, let
$v_{i,0}\in\GG_i$, and let $(b_{i,n})_{n\in\NN}$ 
and $(d_{i,n})_{n\in\NN}$ be  
absolutely summable sequences in $\GG_i$, and let
$(U_{i,n})_{n\in\NN}$ be a sequence in 
$\BP_{\alpha}(\GG_i)$ such that
$(\forall n\in\NN)$ $U_{i,n+1}\succcurlyeq U_{i,n}$.
For every $n\in\NN$, set
\begin{equation}
\delta_n=
\Bigg(\sqrt{\sum_{i= 1}^m\|\sqrt{U_{i,n}} 
L_i\sqrt{U_n}\|^2}\Bigg)^{-1}-1,
\end{equation}
suppose that
\begin{equation}
\zeta_n=\dfrac{1+\delta_n}
{(1+\delta_n)\max\{\|U_n\|,\|U_{1,n}\|,\ldots,\|U_{m,n}\|\}}
\geq\dfrac{1}{2\beta-\varepsilon},
\end{equation}
and let
\begin{equation}
\lambda_n\in\left[\varepsilon,1+(1-\varepsilon)
\bigg(1-\dfrac{1}{2\zeta_n\beta}\bigg)\right].
\end{equation}
Iterate
\begin{equation}
	\label{e:fbpdalgo}
\begin{array}{l}
\text{for}\;n=0,1,\ldots\\
\left\lfloor
\begin{array}{l}
p_n=J_{U_nA}\Big(x_n-U_n
\big(\sum_{i=1}^mL_{i}^*v_{i,n}+C x_n+c_n-z\big)\Big)+a_n\\
y_n=2p_n-x_n\\
x_{n+1}=x_n+\lambda_n(p_n-x_n)\\
\operatorname{for}\;i=1,\ldots, m\\
\left\lfloor
\begin{array}{l}
q_{i,n}=J_{U_{i,n}B_{i}^{-1}}
\Big(v_{i,n}+U_{i,n}\big(L_iy_n- D_{i}^{-1}v_{i,n} 
-d_{i,n}-r_i\big)\Big)+b_{i,n}\\
v_{i,n+1}=v_{i,n}+\lambda_n(q_{i,n}-v_{i,n}).\\
\end{array}
\right.\\[2mm]
\end{array}
\right.\\[2mm]
\end{array}
\end{equation}
Then the following hold for some solution $\overline{x}$ to 
\eqref{e:prob11} and some solution
$(\overline{v}_1,\ldots,\overline{v}_m)$ to \eqref{e:prob12}:
\begin{enumerate}
\item
\label{c:fbpdi}
$x_n\weakly\overline{x}$.
\item
\label{c:fbpdii}
$(\forall i\in\{1,\ldots,m\})$ $v_{i,n}\weakly\overline{v}_i$. 
\item
\label{c:fbpdiii}
Suppose that $C$ is demiregular at $\overline{x}$.
Then $x_n\to\overline{x}$.
\item
\label{c:fbpdiv}
Suppose that, for some $j\in\{1,\ldots,m\}$, $D_j^{-1}$ is 
demiregular at $\overline{v}_j$. Then 
$v_{j,n}\to\overline{v}_j$.
\end{enumerate}
\end{corollary}
\begin{proof} 
Set
$\GGG=\GG_1\oplus\cdots\oplus\GG_m$,
$\KKK=\HH\oplus\GGG$, and
\begin{equation}
\begin{cases}
\label{e:fbpdmm}
\widetilde{A}\hskip -3mm&\colon\KKK\to 2^{\KKK}\colon
(x,v_1,\ldots,v_m)\mapsto (\sum_{i=1}^mL_{i}^*v_i-z+Ax)\\
&\hskip +30mm\times(r_1-L_1x+B_{1}^{-1}v_1)\times\cdots\times
(r_m-L_mx+B^{-1}_{m}v_m)\\
\widetilde{B}\hskip -3mm&\colon\KKK\to\KKK\colon
(x,v_1,\ldots,v_m)\mapsto 
\big(Cx,D^{-1}_1v_1,\ldots,D^{-1}_mv_m\big)\\[2mm]
\widetilde{S}\hskip -3mm&\colon\KKK\to \KKK\colon
(x,v_1,\ldots,v_m)\mapsto
\bigg(\sum_{i=1}^mL_{i}^*v_i,-L_1x,\ldots,-L_mx\bigg).
\end{cases}
\end{equation}
Now, for every $n\in\NN$, define
\begin{equation}
\begin{cases}
\widetilde{U}_n\colon\KKK\to\KKK\colon
(x,v_1,\ldots, v_m)
\mapsto\Big(U_nx,U_{1,n}v_1,\ldots,U_{m,n}v_m\Big)\\[2mm]
\widetilde{V}_n\colon\KKK\to \KKK\colon\\
\quad\quad(x,v_1,\ldots,v_m)\mapsto 
\bigg(U_n^{-1}x-\sum_{i=1}^m L^{*}_iv_i,\big(-L_ix+ 
U_{i,n}^{-1}v_i\big)_{1\leqslant i\leqslant m}\bigg)
\end{cases}
\end{equation}
and 
\begin{equation}
\label{e:fbpd_2}
\begin{cases}
{\widetilde x}_n=(x_n,v_{1,n},\ldots,v_{m,n})\\
{\widetilde y}_n=(p_n,q_{1,n},\ldots,q_{m,n})\\
\widetilde{a}_n=(a_n,b_{1,n},\ldots,b_{m,n})\\
{\widetilde c}_n=(c_n,d_{1,n},\ldots,d_{m,n})\\
{\widetilde d}_n=(U_n^{-1}a_n,
U_{1,n}^{-1}b_{1,n},\ldots, 
U_{m,n}^{-1}b_{m,n})
\end{cases}
\quad\text{and}\quad
\widetilde{b}_n=(\widetilde{S}
+\widetilde{V}_n)\widetilde{a}_n+{\widetilde c}_n
-{\widetilde d}_n.
\end{equation}
It follows from the proof of \cite[Corollary~6.2]{Comb14} that
\eqref{e:fbpdalgo} is equivalent to
\begin{align}
\label{e:fbpd_3}
(\forall n\in\NN)\quad
{\widetilde x}_{n+1}
&={\widetilde x}_n+\lambda_n\Big(J_{{\widetilde V}_n^{-1}
\widetilde{A}}\big({\widetilde x}_n-
{\widetilde V}_{n}^{-1}(\widetilde{B}{\widetilde x}_n+
\widetilde{b}_n)\big)+\widetilde{a}_n-
{\widetilde x}_n\Big),
\end{align}
that the operators $\widetilde{A}$ and $\widetilde{B}$ are 
maximally monotone, and 
$\widetilde{B}$ is $\beta$-cocoercive on $\HHH$.
Furthermore, for every $(\overline{x},\overline{v}) 
\in\zer(\widetilde{A}+\widetilde{B})$, 
$\overline{x}$ solves \eqref{e:prob11} and 
$\overline{v}$ solves \eqref{e:prob12}.
Now set $\rho=1/\alpha+\sqrt{\sum_{i=1}^m\|L_i\|^2}$.
We deduce from the proof of \cite[Corollary~6.2]{Comb14} that
$(\forall n\in\NN)$ $\|\widetilde{V}_{n}^{-1}\|
\leq\zeta_n^{-1}\leq 2\beta-\varepsilon$ and 
$\widetilde{V}_{n+1}^{-1}\succcurlyeq\widetilde{V}_n^{-1}
\in{\EuScript P}_{1/\rho}(\KKK)$.
We observe that \eqref{e:fbpd_3} has the structure of the 
variable metric forward-backward splitting algorithm 
\eqref{e:fb} and that all the conditions of Proposition~\ref{p:fb}
are satisfied. 

\ref{c:fbpdi}\&\ref{c:fbpdii}:
Proposition~\ref{p:fb}\ref{p:fbiii} asserts that there exists
\begin{equation}
\label{e:fbpd_4}
\overline{{\widetilde x}}=(\overline{x},\overline{v}_1,\ldots,
\overline{v}_m)\in\zer(\widetilde{A}+\widetilde{B}) 
\end{equation}
such that ${\widetilde x}_n\weakly\overline{{\widetilde x}}$.

\ref{c:fbpdiii}\&\ref{c:fbpdiv}:
It follows from Proposition~\ref{p:fb}\ref{p:fbii}
that $\widetilde{B}{\widetilde x}_n\to\widetilde{B}
\overline{{\widetilde x}}$. Hence, \eqref{e:fbpdmm},
\eqref{e:fbpd_2}, and \eqref{e:fbpd_4} yield 
\begin{equation}
Cx_n\to C\overline{x} \quad\text{and}\quad 
\big(\forall i\in\{1,\ldots,m\}\big)\quad
D_{i}^{-1}v_{i,n}\to D_{i}^{-1}\overline{v}_i. 
\end{equation}
We derive the results from Definition~\ref{d:demireg} and
\ref{c:fbpdi}--\ref{c:fbpdii} above.
\end{proof}
\begin{remark}
Suppose that $(\forall n\in\NN)$
$\lambda_n\leq 1$. Then Corollary~\ref{c:fbpd} captures 
\cite[Corollary~6.2]{Comb13}.
\end{remark}

{\bf Acknowledgement.} The author thanks his Ph.D. advisor P. L.
Combettes for his guidance during this work, which is part of his
Ph.D. dissertation.

\end{document}